\documentclass[hyp]{amsart}
\usepackage[all]{xy}
\usepackage{hyperref}

\calclayout

\usepackage{amscd,verbatim}
\usepackage{amssymb}
\usepackage[all]{xy}
\usepackage{eucal}
\usepackage{microtype}
\usepackage{xcolor}
\usepackage{hyperref}
\hypersetup{
 colorlinks,
 linkcolor={blue!60!black},
 citecolor={blue!60!black},
 urlcolor={blue!60!black}
}

\title{Zeroth $\mathbb{A}^1$-homology of smooth proper varieties}
\author{Junnosuke Koizumi}
\address{Graduate School of Mathematical Sciences, the University of Tokyo, 3-8-1 Komaba, Meguro-ku, Tokyo 153-8914, Japan}
\email{jkoizumi@ms.u-tokyo.ac.jp}
\keywords{motivic homotopy theory}
\subjclass[2010]{Primary 14F42; Secondary 14E99}

\theoremstyle{plain}
	\newtheorem{theorem}{Theorem}[section]
	\newtheorem{lemma}[theorem]{Lemma}
	
	\newtheorem{corollary}[theorem]{Corollary}
\theoremstyle{definition}

\theoremstyle{remark}

\newcommand{\Set}{{\mathbf {Set}}}
\newcommand{\Ab}{{\mathbf {Ab}}}
\newcommand{\PSh}{{\mathbf{PSh}}}
\newcommand{\Sh}{{\mathbf{Sh}}}
\newcommand{\Data}{{\mathbf{Data}}}

\newcommand{\D}{{\mathbf{D}}}
\newcommand{\sSh}{{\mathbf{sSh}}}
\newcommand{\Sm}{{\mathbf {Sm}}}

\DeclareMathOperator{\Spec}{Spec}
\DeclareMathOperator{\Hom}{Hom}
\DeclareMathOperator{\Aut}{Aut}

\newcommand{\calHom}{{\mathcal{H}\mathrm{om}}}

\newcommand{\del}{\partial}
\newcommand{\br}{{br}}
\newcommand{\ur}{{ur}}

\newcommand{\pr}{{\mathrm{pr}}}
\newcommand{\id}{{\mathrm{id}}}
\newcommand{\conn}{{\mathrm{conn}}}
\newcommand{\A}{{\mathbb{A}^1}}
\newcommand{\bA}{{b\mathbb{A}^1}}
\newcommand{\GTors}{{G\text{-}\mathbf{Tors}}}
\renewcommand{\tilde}{\widetilde}

\newcommand{\Abb}{\mathbb{A}}
\newcommand{\Pbb}{\mathbb{P}}
\newcommand{\Zbb}{\mathbb{Z}}
\newcommand{\Bcal}{\mathcal{B}}
\newcommand{\Ccal}{\mathcal{C}}
\newcommand{\Dcal}{\mathcal{D}}
\newcommand{\Ecal}{\mathcal{E}}
\newcommand{\Fcal}{\mathcal{F}}
\newcommand{\Hcal}{\mathcal{H}}
\newcommand{\Ocal}{\mathcal{O}}
\newcommand{\Vcal}{\mathcal{V}}
\newcommand{\Xcal}{\mathcal{X}}
\newcommand{\Ycal}{\mathcal{Y}}

\begin{document}

\begin{abstract}
We give an explicit formula for the zeroth $\mathbb{A}^1$-homology sheaf of a smooth proper variety.
We also provide a simple proof of a theorem of Kahn-Sujatha which describes hom sets in the birational localization of the category of smooth varieties.
\end{abstract}
\maketitle
\setcounter{tocdepth}{1}
\tableofcontents

\section{Introduction}

We fix a base field $k$.
The sheaf of birational connected components $\pi_0^\br(X)$ of a smooth proper variety $X$ over $k$ was introduced by Asok-Morel \cite{AM} to study $\A$-connected components of $X$.\footnote[1]{It is written $\pi_0^\bA(X)$ in loc.cit.}
It is a Nisnevich sheaf on the category of smooth schemes over $k$ equipped with a morphism of sheaves $X\to \pi_0^\br(X)$ which induces an isomorphism $\pi_0^\br(X)(U)\cong X(k(U))/R$ for a smooth variety $U$ over $k$, where $X(k(U))/R$ is the set of naive $\A$-homotopy classes of morphisms $\Spec k(U)\to X$ \cite[Theorem 6.2.1]{AM}.

On the other hand, the $\A$-homology sheaves $H_n^\A(X)$ of a smooth variety $X$ over $k$ were introduced by Morel \cite{Mor05} as an $\A$-homotopy theoretic analogue of the homology groups of a topological space.
When $X$ is proper over $k$, the canonical morphism $X\to \pi_0^\br(X)$ induces a morphism $H_0^\A(X)\to \Zbb(\pi_0^\br(X))$ by the universal property of $H_0^\A(X)$ (see Lemma \ref{universal_property}).
Asok-Haesemeyer used this to prove that $H_0^\A(X)$ detects existence of a rational point on $X$ \cite[Corollary 2.9]{AH11}.
In this paper we refine this result:

\begin{theorem}[see Theorem \ref{main}]\label{intro1}
	For any smooth proper variety $X$ over a field $k$, the canonical morphism $X\to \Zbb(\pi_0^\br(X))$ induces an isomorphism $H_0^\A(X)\cong \Zbb(\pi_0^\br(X))$.
	In particular, $H_0^\A(X)(k)$ is a free abelian group of rank $\#X(k)/R$.
\end{theorem}

Under resolution of sigularities, Theorem \ref{intro1} was proved by Shimizu \cite[Theorem 4.1]{Shi}, and was also established in unpublished work of Hogadi.
We do not assume resolution of sigularities here.
The key of our proof is the following universal property of $\pi_0^\br(X)$:

\begin{theorem}[see Theorem \ref{pi0_universal}]\label{intro2}
	For any smooth proper variety $X$ over a field $k$, the canonical morphism $X\to \pi_0^\br(X)$ is initial among morphisms from $X$ to $\A$-invariant unramified sheaves over $k$.
	Moreover, if $X\to Y$ is a birational morphism between smooth proper varieties over $k$, then the induced morphism $\pi_0^\br(X)\to \pi_0^\br(Y)$ is an isomorphism.
\end{theorem}

Section 1 contains a proof of this result.
From Theorem \ref{intro2} we also see that the unstable analogue $\pi_0^\A(X)\cong \pi_0^\br(X)$ of Theorem \ref{intro1} does not hold in general; there is an example of a birational morphism $X\to Y$ between smooth proper varieties such that $\pi_0^\A(X)\to \pi_0^\A(Y)$ is not an isomorphism (see \cite[Example 4.8]{BHS}).
We also need the following result:

\begin{theorem}[see Theorem \ref{bga1local}]
	Let $G$ be a strongly $\A$-invariant Nisnevich sheaf of groups on $\Sm_k$ (i.e. $G$ and $H^1(-,G)$ are $\A$-invariant).
	Then its classifying space $\Bcal G$ is $\A$-local.
\end{theorem}

This theorem is stated in \cite{Mor} without proof, but requires a nontrivial computation of the homotopy groups of $\Bcal G$.
This is done in Section 2.
Note that we only use the case where $G$ is abelian and one can give a shorter proof in that case, but we record the general statement for future use.

Theorem \ref{intro2} has another application.
Let $S_b$ denote the class of birational morphisms in the category $\Sm_k^\conn$ of smooth varieties over a field $k$.
Using Theorem \ref{intro2}, we can provide a simple proof of the following result of Kahn-Sujatha {\cite[Theorem 6.6.3]{KS15}}:

\begin{theorem}[see Theorem \ref{ks15main}]
	Let $k$ be a field, $X,U\in \Sm_k^\conn$ and suppose that $X$ is proper.
	Then there is a bijection
	$$
		\Hom_{S_b^{-1}\Sm_k^\conn}(U,X)\cong X(k(U))/R
	$$
	which is compatible with the canonical maps from $X(U)$ to both sides.
\end{theorem}

\subsection*{Notations}
For a scheme $S$, a \emph{smooth scheme} over $S$ is always assumed to be of finite type and separated over $S$.
A \emph{smooth variety} over $S$ is a connected smooth scheme over $S$.
Let $\Sm_S$ (resp. $\Sm_S^\conn$) denote the category of smooth schemes (resp. smooth varieties) over $S$.

The word ``sheaf'' always means sheaf of sets in Nisnevich topology.
Let $\Sh(\Sm_S)$ (resp. $\Ab(\Sm_S)$) denote the cateogory of sheaves (resp. sheaves of abelian groups) on $\Sm_S$.
For $X\in \Sm_S$ and $M\in \Ab(\Sm_S)$, $H^n(X,M)$ denotes the $n$-th Nisnevich cohomology group.
For $X\in \Sm_k$ and $n\geq 0$, we write $X^{(n)}$ for the set of points of codimension $n$ on $X$.

\subsection*{Acknowledgements}
I am grateful to Shuji Saito for supporting me to learn $\Abb^1$-homotopy theory.
I would like to thank Bruno Kahn for pointing out the necessity of a restriction on DVRs.
I would like to thank Aravind Asok for helpful comments on the $\Abb^1$-locality of classifying spaces.
I would also like to thank Anand Sawant and Amit Hogadi for encouraging me to write up the paper.
Finally, I would like to thank the referee for many helpful comments.

\section{Birational sheaves and unramified sheaves}

First we recall the notion of birational sheaves defined in \cite[Definition 6.1.1]{AM}.
A presheaf $S$ on $\Sm_k$ is called a \textit{birational sheaf} if for any $X\in \Sm_k$ the restriction map $S(X)\to \prod_{\eta\in X^{(0)}}S(\eta)$ is a bijection; any birational sheaf is automatically a sheaf \cite[Lemma 6.1.2]{AM}.

\begin{lemma}\label{colliot-thelene}
	Let $S$ be a birational sheaf.
	Then $S$ is $\A$-invariant; i.e. the canonical map $S(X)\to S(\Abb^1_X)$ is bijective for all $X\in \Sm_k$.
\end{lemma}

\begin{proof}
	This is \cite[Appendix A]{KS15} but we spell out the proof for the convinience of readers.
	Let $X\in \Sm_k$ and denote by $q\colon \Pbb^1_X\to X$ the canonical projection.
	It suffices to show that $q^*\colon S(X)\to S(\Pbb^1_X)$ is bijective.
	The injectivity can be seen by taking a section of $q$.
	Let us prove the surjectivity.
	Let $W$ be the blow-up of $\Pbb^1_X\times_X\Pbb^1_X$ along $(\infty,\infty)\cong X$.
	Let $i\colon \Pbb^1_X\to \Pbb^1_X\times_X \Pbb^1_X$ be the morphism $x\mapsto (x,\infty)$, and $\tilde{i}\colon \Pbb^1_X\to W$ the unique lift of $i$.
	Consider the commutative diagram
	$$
	\xymatrix{
		S(\Pbb^1_X\times_X\Pbb^1_X)\ar[r]^-{i^*}\ar[d]		&S(\Pbb^1_X).\\
		S(W)\ar[ur]_-{\tilde{i}^*}
	}
	$$
	Since $i$ has a retraction, $i^*$ is surjective and hence so is $\tilde{i}^*$.
	Now let $\pi\colon W\to \Pbb^2_X$ be the unique morphism extending the open immersion $\Abb^1_X\times_X \Abb^1_X\to \Pbb^2_X; (x,y)\mapsto (x:y:1)$.
	It fits in the following commutative diagram in $\Sm_k$:
	$$
	\xymatrix{
		\Pbb^1_X\ar[r]^-{\tilde{i}}\ar[d]^-{q}		&W\ar[d]^-{\pi}\\
		X\ar[r]^-{(0:1:0)}					&\Pbb^2_X.
	}
	$$
	Since $\pi^*$ is bijective by assumption and $\tilde{i}^*$ is surjective, we get the surjectivity of $q^*$.
\end{proof}

Let $\Fcal_k$ denote the category of finitely generated separable field extensions of $k$.
Let $X$ be a smooth proper variety over $k$ and $F\in \Fcal_k$.
We say that two $F$-points $b,b'\in X(F)$ are \textit{naively $\A$-homotopic} or \textit{$R$-equivalent} if there is a collection $\{\gamma_i\colon \Abb^1_F\to X\}_{i=1}^N$ of morphisms such that $\gamma_1(0)=b$, $\gamma_{i}(1)=\gamma_{i+1}(0)$ and $\gamma_N(1)=b'$.
Let $X(F)/R$ denote the set of $R$-equivalence classes of $F$-points.

\begin{lemma}
	For any smooth proper variety $X$ over $k$, there is a birational sheaf $\pi_0^\br(X)$ together with a morphism $X\to \pi_0^\br(X)$ such that for any smooth variety $U$ there is a bijection $\pi_0^\br(X)(U)\cong X(k(U))/R$ which is compatible with the canonical maps from $X(U)$ to both sides.
\end{lemma}

\begin{proof}
	See \cite[Theorem 6.2.1]{AM}.
\end{proof}

Next we recall from \cite{Mor} the notion of unramified sheaves.
A sheaf $S$ on $\Sm_k$ is called an \textit{unramified sheaf} if the following conditions hold:
\begin{enumerate}
	\item For any $X\in \Sm_k^\conn$ the restriction map $S(X)\to S(k(X))$ is injective.
	\item For any $X\in \Sm_k^\conn$ we have $S(X)=\bigcap_{x\in X^{(1)}}S(\Ocal_{X,x})$ as subsets of $S(k(X))$.
\end{enumerate}
Let $\Sh^\ur(\Sm_k)$ denote the full subcategory of $\Sh(\Sm_k)$ consisting of unramified sheaves.

Let $\Vcal_k$ denote the class of all pairs $(F,v)$ where $F\in \Fcal_k$ and $v$ is a discrete valuation on $F$ such that there are some $X\in \Sm_k$ and $x\in X^{(1)}$ such that $\Ocal_v \cong\Ocal_{X,x}$.
An \textit{unramified $\Fcal_k$-datum} is a functor $S\colon \Fcal_k\to \Set$ together with a subset $S(\Ocal_v)\subset S(F)$ and a specialization map $s_v\colon S(\Ocal_v)\to S(k(v))$ for each $(F,v)\in \Vcal_k$ satisfying some compatibility conditions (see \cite[Definition 2.9]{Mor} for details).
A morphism of $\Fcal_k$-data $f\colon S\to S'$ is a natural transformation satisfying $f_F(S(\Ocal_v))\subset S'(\Ocal_v)$ for every $(F,v)\in \Vcal_k$ and compatible with specialization maps.
Let $\Data^\ur$ denote the category of unramified $\Fcal_k$-data.

\begin{lemma}\label{unramified_datum}
	The restriction functor $\Sh^\ur(\Sm_k)\to \Data^\ur$ is an equivalence of categories.
\end{lemma}

\begin{proof}
	See \cite[Theorem 2.11]{Mor}.
\end{proof}

\begin{theorem}\label{pi0_universal}
	For any smooth proper variety $X$ over $k$, the morphism $X\to \pi_0^\br(X)$ is initial among morphisms from $X$ to $\A$-invariant unramified sheaves.
	Moreover, if $X\to Y$ is a birational morphism between smooth proper varieties over $k$, then the induced morphism $\pi_0^\br(X)\to \pi_0^\br(Y)$ is an isomorphism.\end{theorem}

\begin{proof}
	Firstly, $\pi_0^\br(X)$ itself is $\A$-invariant by Lemma \ref{colliot-thelene} and unramified by definition.
	For any $\A$-invariant unramified sheaf $S$ and $a\in S(X)$, we prove that there exists a unique morphism of sheaves $f\colon \pi_0^\br(X)\to S$ which makes the following diagram commute:
	$$
	\xymatrix{
		X\ar[r]^-a\ar[d]		&S.\\
		\pi_0^\br(X)\ar[ur]_-f
	}
	$$
	Since $S$ is unramified, commutativity of the above diagram is equivalent to that of
	$$
	\xymatrix{
		&X(F)\ar[r]^-{a_F}\ar[d]		&S(F)\\
		X(F)/R\ar@{=}[r]&\pi_0^\br(X)(F)\ar[ur]_-{f_F}
	}
	$$
	for each $F\in \Fcal_k$.
	This proves the uniqueness of $f_F$ and hence the uniqueness of $f$.
	
	We prove the existence of $f$.
	First we construct $f_F$ for each $F\in \Fcal_k$ so that the above diagram becomes commutative.
	Consider the following commutative diagram for $\varepsilon=0,1$:
	$$
	\xymatrix{
		X(\Abb^1_F)\ar[r]^-{a_{\Abb^1_F}}\ar[d]^-{i_\varepsilon^*}		&S(\Abb^1_F)\ar[d]^-{i_\varepsilon^*}\\
		X(F)\ar[r]^-{a_F}									&S(F).
	}
	$$
	Here $i_\varepsilon$ denote the $F$-valued point $F \to \Abb^1_F$ with coordinate $\varepsilon$.
	Since $S$ is $\A$-invariant we have $i_0^*=i_1^*\colon S(\Abb^1_F)\to S(F)$.
	This shows that $a_F$ descends to $R$-equivalence classes and gives a map $f_F\colon \pi_0^\br(X)(F)\to S(F)$ for each $F\in \Fcal_k$.
	
	Next we have to show that $\{f_F\}_{F\in \Fcal_k}$ extends to a morphism of sheaves.
	Since $\pi_0^\br(X)$ and $S$ are both unramified, it suffices to check that $\{f_F\}_{F\in \Fcal_k}$ gives a morphism in $\Data^\ur$ by Lemma \ref{unramified_datum}.
	Let $(F,v)\in \Vcal_k$.
	Consider the following diagram:
	$$
	\xymatrix{
		X(\Ocal_v)\ar[r]\ar@/^15pt/[rr]^-{a_{\Ocal_v}}\ar[d]	&\pi_0^\br(X)(\Ocal_v)\ar[d]^-{\cong}					&S(\Ocal_v)\ar@{^(->}[d]\\
		X(F)\ar@{->>}[r]			&\pi_0^\br(X)(F)\ar[r]^-{f_F}			&S(F).
	}
	$$
	The left square and the total rectangle are commutative.
	By the valuative criterion of properness, the left vertical map is bijective.
	Hence we have $f_F(\pi_0^\br(X)(\Ocal_v))\subset S(\Ocal_v)$, and the map $X(\Ocal_v)\to \pi_0^\br(X)(\Ocal_v)$ is surjective.
	Consider the following diagram:
	$$
	\xymatrix{
		X(\Ocal_v)\ar@{->>}[r]\ar[d]		&\pi_0^\br(X)(\Ocal_v)\ar[r]^-{f_F}\ar[d]		&S(\Ocal_v)\ar[d]\\
		X(k(v))\ar[r]				&\pi_0^\br(X)(k(v))\ar[r]^-{f_{k(v)}}				&S(k(v)).
	}
	$$
	Again the left square and the total rectangle are commutative.
	Hence the right square is also commutative and we get a morphism in $\Data^\ur$.
	This proves the first statement.
	
	Now we prove the second statement.
	By the universal property we have just proved, it suffices to show that $S(Y)\to S(X)$ is bijective for any $\A$-invariant unramified sheaf $S$.
	Let $F$ be the function field of $X$.
	Considering $S(X)$ and $S(Y)$ as subsets of $S(F)$, we have an inclusion $S(Y)\subset S(X)$.
	To prove $S(X)\subset S(Y)$, it suffices to prove $S(X)\subset S(\Ocal_{Y,y})$ for any $y\in Y^{(1)}$.
	By the valuative criterion of properness, there is a morphism $\Spec \Ocal_{Y,y}\to X$ extending $\Spec F\to X$.
	This implies the required inclusion.
\end{proof}

\section{$\Abb^1$-locality of classifying spaces}

In this section we work over a noetherian base scheme $S$ of finite Krull dimension.
We prove an extension of \cite[Section 4 Proposition 1.16]{MV99} and use it to prove Theorem \ref{bga1local}.
Much of the discussion here goes back to work of Giraud and is folklore in the literature on stacks and homotopy theory.
For related results in a different formulation, see \cite[Chapter 9]{Jardine_LHT}.
We also note that most of the results here are not special to the  Nisnevich topology on smooth schemes.

We write $\sSh(\Sm_S)$ for the category of simplicial sheaves on $\Sm_S$ equipped with the simplicial model structure \cite[Section 2 Definition 1.2]{MV99}, and $\Hcal_s(S)$ for its homotopy category.
For a sheaf of groups $G$ on $\Sm_S$ and $\Xcal\in \sSh(\Sm_S)$, we write $P(\Xcal, G)$ for set of isomorphism classes of Nisnevich $G$-torsors on $\Xcal$.

\begin{lemma}\label{bg_universality}
	Let $G$ be a sheaf of groups on $\Sm_S$.
	Then there is a fibrant simplicial sheaf $\Bcal G\in \sSh(\Sm_S)$ together with a $G$-torsor $\Ecal G$ on it such that 
	$$
	\Hom_{\sSh(\Sm_S)}(\Xcal, \Bcal G)\to P(\Xcal,G);~f\mapsto f^*\Ecal G
	$$
	induces a bijection $\Hom_{\Hcal_s(S)}(\Xcal,\Bcal G)\cong P(\Xcal,G)$.
\end{lemma}

\begin{proof}
	See \cite[Section 4 Proposition 1.15]{MV99}.
\end{proof}

Let us fix such $\Bcal G$ and $\Ecal G$.
For any $U\in \Sm_S$ the simplicial set $\Bcal G(U)$ is a Kan complex since $\Bcal G$ is fibrant.
We want to describe this Kan complex explicitly.

\begin{lemma}\label{truncated}
	We have $\pi_n(\Bcal G(U),p)=0~(n\geq 2)$ for any vertex $p\in \Bcal G(U)_0$.
\end{lemma}

\begin{proof}
	It suffices to prove that the restriction map from $[\Delta^{n+1},\Bcal G(U)]$ to $[\del\Delta^{n+1},\Bcal G(U)]$ is surjective for $n\geq 2$.
	By Lemma \ref{bg_universality} this is equivalent to saying that any $G$-torsor on $\del\Delta^{n+1}\times U$ can be extended to $\Delta^{n+1}\times U$.
	Let $\Ycal\to \del\Delta^{n+1}\times U$ be a $G$-torsor, and consider the evaluation morphism $\mathrm{ev}\colon \del\Delta^{n+1}\times \calHom_{\del\Delta^{n+1}}(\del\Delta^{n+1},\Ycal)\to \Ycal$.
	For any $V\in \Sm_S$ and a point $x\in V$, $\mathrm{ev}_x\colon \del\Delta^{n+1}\times \Hom_{\del\Delta^{n+1}}(\del\Delta^{n+1},\Ycal_x)\to \Ycal_x$ is an isomorphism since any $G_x$-torsor on $\del\Delta^{n+1}\times U_x$ is trivial.
	Hence $\Ycal\cong \del\Delta^{n+1}\times\calHom_{\del\Delta^{n+1}}(\del\Delta^{n+1},\Ycal)$ and the claim is trivial.
\end{proof}

Thus the canonical map $\Bcal G(U)\to \mathrm{N}(\Pi_1(\Bcal G(U)))$ is a homotopy equivalence, where $\Pi_1$ denotes the fundamental groupoid and $\mathrm{N}$ denotes its nerve.
Let $\GTors_U$ denote the groupoid of $G$-Torsors on $U$.
We construct a functor $\Phi\colon \Pi_1(\Bcal G(U))\to \GTors_U$ as follows.
For an object $p\in \Bcal G(U)_0$, we define $\Phi(p)=p^*\Ecal G$.
Let $e\in \Bcal G(U)_1$ an edge from $p_0\in \Bcal G(U)_0$ to $p_1\in \Bcal G(U)_0$.
Then $e^*\Ecal G$ is a $G$-torsor on $\Delta^1\times U$ which restricts to $p_\varepsilon^*\Ecal G$ over $\Delta^{\{i\}}\times U$ for $i=0,1$.
Using the next lemma, we get a unique isomorphism of $G$-torsors $\rho_e\colon \Delta^1\times p_0^*\Ecal G\to e^*\Ecal G$ which restricts to the identity over $\Delta^{\{0\}}\times U$.
We define $\Phi(e)$ to be the restriction of $\rho_e$ over $\Delta^{\{1\}}\times U$.
This indeed gives a functor $\Phi\colon \Pi_1(\Bcal G(U))\to \GTors_U$, again by the next lemma.

\begin{lemma}
	Let $\Ycal\to \Delta^n\times U$ be a $G$-torsor and $\Ycal_0$ its restriction to $\Delta^{\{0\}}\times U$.
	Then there exists a unique isomorphism of $G$-torsors $\rho\colon \Delta^n\times \Ycal_0\to \Ycal$ which restricts to the identity over $\Delta^{\{0\}}\times U$.
\end{lemma}

\begin{proof}
	By the same argument as in the proof of Lemma \ref{truncated}, we may assume $\Ycal\cong \Delta^n\times \Ycal_0$.
	It suffices to show that if $\rho\colon \Delta^n\times \Ycal_0\to \Delta^n\times \Ycal_0$ is an isomorphism which restricts to the identity on $\Delta^{\{0\}}\times U$ then $\rho=\id$.
	This can be checked by taking stalks.
\end{proof}

The following lemma is easy.

\begin{lemma}\label{groupoid_lemma}
	Let $\Ccal$ and $\Dcal$ be groupoids and $\varphi\colon \Ccal\to \Dcal$ a functor.
	Then $\varphi$ is an equivalence of groupoids if and only if the following hold:
	\begin{enumerate}
		\item $\varphi$ is essentially surjective and conservative.
		\item For any $c\in \Ccal$, $\Aut_\Ccal(c)\to \Aut_\Dcal(\varphi(c))$ is an isomorphism.
	\end{enumerate}
\end{lemma}

\begin{lemma}
	$\Phi\colon \Pi_1(\Bcal G(U))\to \GTors_U$ is an equivalence of groupoids.
\end{lemma}

\begin{proof}
	We check the conditions in Lemma \ref{groupoid_lemma}.
	(i) is an immediate consequence of Lemma \ref{bg_universality}.
	Take any vertex $p\in \Bcal G(U)_0$.
	If $e\in \Bcal G(U)_1$ is an edge from $p$ to $p$ such that $\Phi(e)=\id_{p^*\Ecal G}$, then the pullback of $\Ecal G$ by $e\colon S^1\times U\to \Bcal G$ is a trivial $G$-torsor.
	Thus there is some homotopy from $e\colon S^1\to \Bcal G(U)$ to a constant map.
	This shows that $[e]=[\id_p]$ in $\Pi_1(\Bcal G(U))$, and hence $\Aut_{\Pi_1(\Bcal G(U))}(p)\to \Aut_{\GTors_U}(p^*\Ecal G)$ is injective.
	Let $\rho\colon p^*\Ecal G\to p^*\Ecal G$ be an automorphism of $G$-torsors.
	We define $\Ycal$ to be the coequalizer of
	$$
	\xymatrix{
		 p^*\Ecal G\ar@<0.5ex>[r]^-{(0,\id)} \ar@<-0.5ex>[r]_-{(1,\rho)} &\Delta^1\times p^*\Ecal G
	}
	$$
	which is a $G$-torsor on $S^1\times U$.
	Let $f\colon S^1\to \Bcal G(U)$ be a loop classifying $\Ycal$ and $p'$ its endpoint.
	Since $p'^*\Ecal G$ is isomorphic to $p^*\Ecal G$, there is an edge $g\in \Bcal G(U)_1$ from $p$ to $p'$.
	Then $\Phi([g]^{-1}\circ [f]\circ [g])=\rho$ by construction and hence $\Aut_{\Pi_1(\Bcal G(U))}(p)\to \Aut_{\GTors_U}(p^*\Ecal G)$ is surjective.
\end{proof}

We obtain an extension of \cite[Section 4 Proposition 1.16]{MV99}:

\begin{corollary}\label{MV1.16}
	Let $G$ be a sheaf of groups on $\Sm_S$ and $U\in \Sm_S$.
	Then there is a canonical homotopy equivalence $\Bcal G(U)\to \mathrm{N}(\GTors_U)$ sending $p\in \Bcal G(U)_0$ to $p^*\Ecal G$.
	In particular, the homotopy groups of $\Bcal G(U)$ are given by
	$$
		\pi_n(\Bcal G(U),p)\cong\begin{cases}
			H^1(U,G)		&(n=0)\\
			\Aut_{\GTors_U}(p^*\Ecal G)		&(n=1)\\
			0			&(n\geq 2).
		\end{cases}
	$$
\end{corollary}

We say that a sheaf of groups $G$ on $\Sm_S$ is \textit{strongly $\A$-invariant} if $G$ and $H^1(-,G)$ are $\A$-invariant.
An object $\Xcal\in \sSh(\Sm_S)$ is said to be \textit{$\A$-local} if the map
$$
	\Hom_{\Hcal_s(S)}(\Ycal,\Xcal)\to \Hom_{\Hcal_s(S)}(\Ycal\times \A,\Xcal)
$$
induced by the canonical projection is bijective for every $\Ycal\in \sSh(\Sm_S)$.

\begin{theorem}\label{bga1local}
	Let $G$ be a strongly $\A$-invariant sheaf of groups on $\Sm_S$.
	Then $\Bcal G$ is $\A$-local.
\end{theorem}

\begin{proof}
	By \cite[Section 2 Proposition 3.19]{MV99} it suffices to show that $\pi_n(\Bcal G(U),p)\to \pi_n(\Bcal G(\Abb^1_U),r^*p)$ is bijective for any $U\in \Sm_S$ and $p\in \Bcal G(U)_0$, where $r\colon \Abb^1_U\to U$ is the canonical projection.
	The cases $n=0$ and $n\geq 2$ follows from Corollary \ref{MV1.16} and the $\A$-invariance of $H^1(-,G)$.
	It remains to show that $\Aut_{\GTors_U}(p^*\Ecal G)\to \Aut_{\GTors_{\Abb^1_U}}(r^*p^*\Ecal G)$ is an isomorphism.
	Take a Nisnevich covering $q\colon V\to U$ such that $q^*p^*\Ecal G$ is a trivial $G$-torsor.
	Let $\varphi\colon \pr_1^*q^*p^*\Ecal G\xrightarrow{\cong}\pr_2^*q^*p^*\Ecal G$ be the descent data for $\Ecal G$; it is an isomorphism of $G$-torsors on $V\times_UV$.
	Now a choice of an isomorphism $\sigma\colon q^*p^*\Ecal G\cong V\times G$ enables us to identify $\Aut_{\GTors_V}(q^*p^*\Ecal G)$ with $G(V)$ and $\varphi$ with an element $h\in G(V\times_UV)$.
	By Nisnevich descent for $G$-torsors, we obtain an isomorphism
	$$
		\Aut_{\GTors_U}(p^*\Ecal G)\cong \{g\in G(V)\mid h(\pr_1^*g)=(\pr_2^*g)h\}.
	$$
	On the other hand, $\sigma$ also induces an identification $q'^*r^*p^*\Ecal G\cong \Abb^1_V\times G$ where $q'\colon \Abb^1_V\to \Abb^1_U$ is the morphism induced by $q$.
	We thus get a similar isomorphism
	$$
		\Aut_{\GTors_{\Abb^1_U}}(r^*p^*\Ecal G)\cong \{g\in G(\Abb^1_V)\mid h'(\pr_1^*g)=(\pr_2^*g)h'\}.
	$$
	where $h'$ is the image of $h$ in $G(\Abb^1_V\times_{\Abb^1_U}\Abb^1_V)$.
	The morphism under consideration is identified with the morphism
	$$
		\{g\in G(V)\mid h(\pr_1^*g)=(\pr_2^*g)h\}\to \{g\in G(\Abb^1_V)\mid h'(\pr_1^*g)=(\pr_2^*g)h'\}
	$$
	induced by the canonical projection.
	Since $G$ is $\A$-invariant, this is an isomorphism.
\end{proof}

We now return to the theory over a field $k$.

\begin{corollary}\label{a1_unr}
	Any strongly $\A$-invariant sheaf of groups on $\Sm_k$ is unramified.
\end{corollary}

\begin{proof}
	If $G$ is a strongly $\A$-invariant sheaf of groups on $\Sm_k$, then $\Bcal G$ is $\A$-local by Lemma \ref{bga1local}.
	Hence by the proof of \cite[Theorem 6.1]{Mor}, $G=\pi_1^\A(\Bcal G,\ast)$ is unramified.
\end{proof}

\section{Main result}

We recall the definition of $\A$-homology sheaves.
A sheaf $M$ of abelian groups on $\Sm_k$ is called \textit{strictly $\A$-invariant} if its cohomology presheaves $H^n(-,M)~(n\geq 0)$ are all $\A$-invariant.
Let $\Ab^\A(\Sm_k)$ denote the full subcategory of $\Ab(\Sm_k)$ consisting of strictly $\A$-invariant sheaves.
Let $\D(\Sm_k)$ denote the unbounded derived category of $\Ab(\Sm_k)$.
An object $C\in \D(\Sm_k)$ is called $\A$-local if for any $D\in \D(\Sm_k)$ the map
$$
	\Hom_{\D(\Sm_k)}(D,C)\to \Hom_{\D(\Sm_k)}(D\otimes \Zbb(\A),C)
$$
is bijective.
For example, a sheaf of abelian groups (viewed as an object of $\D(\Sm_k)$) is $\A$-local if and only if it is strictly $\A$-invariant.
Morel's $\A$-derived category $\D_\A(k)$ is the full triangulated subcategory of $\D(\Sm_k)$ consisting of $\A$-local complexes.
The inclusion functor $\D_\A(k)\to \D(\Sm_k)$ admits a left adjoint $L_\A$ (see \cite[Corollary 6.19]{Mor}).
We say that $C\in \D(\Sm_k)$ is \textit{$(-1)$-connected} if $H_nC=0$ for all $n<0$.

\begin{theorem}\label{connectivity}
	If $C\in \D(\Sm_k)$ is $(-1)$-connected, then $L_\A C$ is also $(-1)$-connected.
\end{theorem}

\begin{proof}
	See \cite[Theorem 6.22]{Mor}.
\end{proof}

\begin{corollary}
	An object $C\in \D(\Sm_k)$ is $\A$-local if and only if its homology sheaves are all strictly $\A$-invariant.
\end{corollary}

\begin{proof}
	This is well-known, but we recall the proof for the convinience of readers.
	The `if' part is clear.
	To prove the `only if' part, it suffices to show that if $C\in \D(\Sm_k)$ is $\A$-local then its truncation $\tau_{\geq n}C$ is also $\A$-local for each $n\in \Zbb$.
	By Theorem \ref{connectivity}, $L_\A\tau_{\geq n} C$ is $(n-1)$-connected.
	Thus the canonical morphism $L_\A\tau_{\geq n} C\to L_\A C\cong C$ factors through $\tau_{\geq n}C$.
	Since the composition $\tau_{\geq n}C\to L_\A\tau_{\geq n} C\to \tau_{\geq n}C$ is the identity, it follows that $\tau_{\geq n}C$ is $\A$-local.
\end{proof}

This implies that the canonical t-structure on $\D(\Sm_k)$ restricts to a t-structure on $\D_\A(k)$ whose heart is $\Ab^\A(\Sm_k)$.
The \textit{$n$-th $\A$-homology sheaf} $H_n^\A(X)$ of $X\in \Sm_k$ is defined to be the $n$-th homology sheaf of $L_\A\Zbb(X)$, which is strictly $\A$-invariant and vanishes for $n<0$.
There is a canonical morphism of sheaves $X\to H_0^\A(X)$.

\begin{lemma}\label{universal_property}
	Let $X\in \Sm_k$.
	Then the canonical morphism $X\to H_0^\A(X)$ is initial among morphisms from $X$ to strictly $\A$-invariant sheaves of abelian groups.
\end{lemma}

\begin{proof}
	This is originally due to Morel, and a proof can be found in \cite[Lemma 3.3]{Aso12}.
	We recall the proof for the convinience of readers.
	For any $M\in \Ab^\A(\Sm_k)$ we have a sequence of natural isomorphisms
	\begin{align*}
		\Hom_{\Sh(\Sm_k)}(X,M)	&\cong \Hom_{\D(\Sm_k)}(\Zbb(X), M)\\
								&\cong \Hom_{\D_\A(k)}(L_\A\Zbb(X), M)\\
								&\cong \Hom_{\D_\A(k)}(\tau_{\leq 0} L_\A\Zbb(X),M).
	\end{align*}
	Since $L_\A\Zbb(X)$ is $(-1)$-connected, we have $\tau_{\leq 0} L_\A\Zbb(X)\cong H_0^\A(X)$.
	Thus the last group is isomorphic to $\Hom_{\Ab^\A(\Sm_k)}(H_0^\A(X),M)$.
\end{proof}

\begin{lemma}\label{br_str_a1}
	Any birational sheaf of abelian groups is strictly $\A$-invariant.
\end{lemma}

\begin{proof}
	Let $M$ be a birational sheaf of abelian groups.
	Then $M$ is $\A$-invariant by Lemma \ref{colliot-thelene}.
	Moreover we have $H^n(-,M)=0~(n>0)$ since the \v{C}ech cohomology groups vanish for $n>0$ by birationality (see \cite[Lemma 2.4]{AH11} for details).
	Therefore $M$ is strictly $\A$-invariant.
\end{proof}

\begin{theorem}\label{main}
	For any smooth proper variety $X$ over a field $k$, the canonical morphism $X\to \Zbb(\pi_0^\br(X))$ induces an isomorphism $H_0^\A(X)\cong \Zbb(\pi_0^\br(X))$.
	In particular, $H_0^\A(X)(k)$ is a free abelian group of rank $\#X(k)/R$.
\end{theorem}

\begin{proof}
	By Lemma \ref{br_str_a1} the sheaf $\Zbb(\pi_0^\br(X))$ is strictly $\A$-invariant.
	By Theorem \ref{pi0_universal} and Corollary \ref{a1_unr}, the canonical morphism $X\to \Zbb(\pi_0^\br(X))$ is initial among morphisms from $X$ to strictly $\A$-invariant sheaves of abelian groups.
	Hence by Lemma \ref{universal_property} we get the required isomorphism.
\end{proof}

Next we give another application of Theorem \ref{pi0_universal}.
Let $S_b$ denote the class of birational morphisms in $\Sm_k^\conn$.
Then the category of birational sheaves is equivalent to the category $\PSh(S_b^{-1}\Sm_k^\conn)$ of presheaves on $S_b^{-1}\Sm_k^\conn$.
For any smooth variety $X$ over $k$, define $h_X\in \PSh(S_b^{-1}\Sm_k^\conn)$ to be the presheaf represented by $X$.
Let $\tilde{h}_X$ denote the corresponding birational sheaf.
There is a canonical morphism $X\to \tilde{h}_X$.

\begin{lemma}\label{universal2}
	Let $X$ be a smooth variety over $k$.
	Then the canonical morphism $X\to \tilde{h}_X$ is initial among morphisms from $X$ to birational sheaves.
\end{lemma}

\begin{proof}
	Let $X\to S$ be a morphism to a birational sheaf on $\Sm_k$.
	Write $S'$ for the presheaf on $S_b^{-1}\Sm_k^\conn$ corresponding to $S$.
	Then the claim follows from the sequence of natural bijections
	\begin{align*}
		\Hom_{\Sh(\Sm_k)}(X,S)			&\cong S(X) \cong S'(X)\\
										&\cong \Hom_{\PSh(S_b^{-1}\Sm_k^\conn)}(h_X, S')\\
										&\cong \Hom_{\Sh(\Sm_k)}(\tilde{h}_X,S)
	\end{align*}
	where the first and the third bijection is due to the Yoneda lemma.
\end{proof}

We obtain a simple proof of the following result of Kahn-Sujatha \cite[Theorem 6.6.3]{KS15}:

\begin{theorem}\label{ks15main}
	Let $X,U\in \Sm_k^\conn$ and suppose that $X$ is proper.
	Then there is a bijection
	$$
		\Hom_{S_b^{-1}\Sm_k^\conn}(U,X)\cong X(k(U))/R
	$$
	which is compatible with the canonical maps from $X(U)$ to both sides.
\end{theorem}

\begin{proof}
	By Theorem \ref{pi0_universal} and Lemma \ref{colliot-thelene}, the canonical morphism $X\to \pi_0^\br(X)$ is initial among morphisms from $X$ to birational sheaves.
	Hence by Lemma \ref{universal2} there is a unique isomorphism $\tilde{h}_X\cong\pi_0^\br(X)$ which is compatible with the canonincal morphisms from $X$ to both sides.
	Evaluating at $U\in \Sm_k^\conn$ we obtain the required bijection.
\end{proof}

\end{document}